\documentclass{amsart}

\usepackage{amsthm,amssymb,latexsym,amsmath}
\usepackage[all]{xy}
\usepackage[english]{babel}
\usepackage[utf8x]{inputenc}
\usepackage{xcolor}
\usepackage[colorlinks,linkcolor=blue,citecolor=blue,filecolor=blue,urlcolor=blue]{hyperref}

\newcommand{\p}[1]{{\mathbb{P}^{#1}}}

\newcommand{\op}[1]{{\mathcal O}_{\mathbb{P}^{#1}}}

\newcommand{\cala}{{\mathcal A}}
\newcommand{\calb}{{\mathcal B}}

\newcommand{\cale}{{\mathcal E}}

\newcommand{\calh}{{\mathcal H}}
\newcommand{\cali}{{\mathcal I}}

\newcommand{\calm}{{\mathcal M}}
\newcommand{\caln}{{\mathcal N}}
\newcommand{\calo}{{\mathcal O}}

\newcommand{\simto}{\stackrel{\sim}{\to}}

\DeclareMathOperator{\rk}{{rk}}

\newlength{\rrrr}

\newcommand{\intoo}[1]{\:
\xymatrix@1{\ar@{^(->}[r]^{#1}&}\:}
\newcommand{\ontoo}[1]{\:
\xymatrix@1{\ar@{->>}[r]^{#1}&}\:}

\newtheorem{theorem}{Theorem}

\newtheorem{proposition}[theorem]{Proposition}

\newtheorem{remark}[theorem]{Remark}

\begin{document}

\title{Construction of stable rank 2 vector bundles on $\mathbb{P}^3$ via symplectic bundles}

\author{Alexander Tikhomirov}
\address{Faculty of Mathematics\\
National Research University  
Higher School of Economics\\
6 Usacheva Street\\ 
119048 Moscow, Russia}
\email{astikhomirov@mail.ru}

\author{Sergey Tikhomirov}
\address{Faculty of Mathematics and Physics,
Yaroslavl State Pedagogical University named after K.D.Ushinskii,
108 Respublikanskaya Street,
150000 Yaroslavl, Russia\\ 
${\ \ \ \ }$Koryazhma Branch of Northern (Arctic)
Federal University named after M.V.Lomonosov\\
9 Lenin Avenue\\
165651 Koryazhma, Russia}
\email{satikhomirov@mail.ru}

\author{Danil Vasiliev}
\address{Faculty of Mathematics\\
National Research University  
Higher School of Economics\\
6 Usacheva Street\\ 
119048 Moscow, Russia}
\email{davasilev@edu.hse.ru}

\begin{abstract}
In this article we study the Gieseker-Maruyama moduli spaces $\mathcal{B}(e,n)$ of stable rank 2 algebraic vector bundles with Chern classes $c_1=e\in\{-1,0\},\ c_2=n\ge1$ on the projective space $\mathbb{P}^3$. We construct two new infinite series $\Sigma_0$ and $\Sigma_1$ of irreducible components of the
spaces $\mathcal{B}(e,n)$ for $e=0$ and $e=-1$, respectively. General bundles of these components are obtained as cohomology sheaves of monads, the middle term of which is a rank 4 symplectic instanton bundle in case $e=0$, respectively, twisted symplectic bundle in case $e=-1$. We show that the series $\Sigma_0$ contains components for all big enough values of $n$ (more precisely, at least for $n\ge146$). $\Sigma_0$ yields the next example, after the series of instanton components, of an infinite series of components of $\mathcal{B}(0,n)$ satisfying this property.



\end{abstract}

\maketitle

\section{Introduction}\label{section 1}

For $e\in\{-1,0\}$ and $n\in\mathbb{Z}_+$ let $\mathcal{B}(e,n)$ be the Gieseker-Maruyama moduli space of stable rank 2 algebraic vector bundles with Chern classes $c_1=e,\ c_2=n$ on the projective space $\mathbb{P}^3$. R.~Hartshorne \cite{H-vb} showed that $\mathcal{B}(e,n)$ is a quasi-projective scheme, nonempty for arbitrary $n\ge1$ in case $e=0$ and, respectively, for even $n\ge2$ in case $e=-1$, and the deformation theory predicts that each irreducible component of $\calb(e,n)$ has dimension at least $8n-3+2e$.  

In case $e=0$ it is known by now (see \cite{H-vb}, \cite{ES}, \cite{B1}, \cite{CTT}, \cite{JV}, \cite{T1}, \cite{T2}) that the scheme $\mathcal{B}(0,n)$ contains an irreducible component $I_n$ of expected dimension $8n-3$, and this component is the closure of the smooth open subset of $I_n$ constituted by the so-called  mathematical instanton vector bundles. Historically, $\{I_n\}_{n\ge1}$ was the first known infinite series of irreducible components of $\mathcal{B}(0,n)$ having the expected dimension $\dim I_n=8n-3$. In \cite[Ex. 4.3.2]{H-vb} R.~Hartshorne constructed a first infinite series $\{\mathcal{B}_0(-1,2m)\}_{m\ge1}$ of irreducible components $\mathcal{B}_0(-1,2m)$ of $\mathcal{B}(-1,2m)$  having the expected dimension $\dim\mathcal{B}_0(-1,2m)=16m-5$. 

The other infinite series of families of vector bundles of dimension $3k^2+10k+8$ from 
$\calb(0,2k+1)$ was constructed in 1978 by W.~Barth and K.~Hulek \cite{BH}, and G.~Ellingsrud and S.~A.~Str\o mme in \cite[(4.6)-(4.7)]{ES} showed that these families are open subsets of irreducible componens distint from the instanton components $I_{2k+1}$. 
Later in 1985-87 V.~K.~Vedernikov \cite{V1} and \cite{V2} constructed two infinite series of families of bundles from
$\calb(0,n)$, and one infinite family of bundles from $\calb(-1,2m)$. A more general series of rank 2 bundles depending on triples of integers $a,b,c$, appeared in 1984 in the paper of A. Prabhakar Rao \cite{Rao}. Soon after that, in 1988, L.~Ein \cite{Ein} independently studied these bundles  and proved that they constitute open parts of irreducible components of $\calb(e,n)$ for both $e=0$ and $e=-1$.

A new progress in the description of the spaces $\calb(0,n)$
was achieved in 2017 by J.~Almeida, M.~Jardim, A.~Tikhomirov and
S.~Tikhomirov in \cite{AJTT}, where they constructed a new infinite series of irreducible components $Y_a$ of the spaces $\calb(0,1+a^2)$ for $a\in\{2\}\cup\mathbb{Z}_{\ge4}$. These components have dimensions 
$\dim Y_a=4\binom{a+3}{3}-a-1$ which is larger than expected.
General bundles from these components are obtained as cohomology bundles of rank 1 monads, the middle term of which
is a rank 4 symplectic instanton with $c_2=1$, and the lefthand and the righthand terms are $\op3(-a)$ and $\op3(a)$, respectively.

The aim of present article is to provide two new infinite series of irreducible components $\calm_n$ of $\calb_(e,n)$, one for $e=0$ and another for $e=-1$ which in some sense generalizes the above construction from \cite{AJTT}.  
Namely, in case $e=0$ we construct an infinite series $\Sigma_0$  of irreducible components $\calm_n$ of $\calb(0,n)$, such that a general bundle of $\calm_n$ is a cohomology bundle of a monad of the type similar to the above, the middle term of which is a rank 4 symplectic instanton with arbitrary second Chern class. The first main result of the article, Theorem \ref{Thm A}, states that the
series $\Sigma_0$ contains components $\calm_n$ for all big
enough values of $n$ (more precisely, at least for $n\ge146$). The series $\Sigma_0$ is a first example, besides the instanton series $\{I_n\}_{n\ge1}$, of the series with this property. (For all the other series mentioned above the question whether they contain components with all big enough values of second Chern class $n$ is open.) 

In case $e=-1$ we construct in a similar way an infinite series $\Sigma_1$ of irreducible components $\calm_n$ of $\calb(-1,n)$, such that a general bundle of $\calm_n$ is a cohomology bundle of a monad of the type similar to the above, in which the lefthand and the righthand terms are $\op3(-a-1)$ and $\op3(a)$, respectively, and the middle term is a twisted symplectic rank 4 bundle with first Chern
class -2. The second main result of the article, Theorem \ref{Thm B}, states that $\Sigma_1$ contains components  $\calm_n$ asymptotically for almost all big enough values of
$n$. (A precise statement about the behaviour of the set
of values of $n$ for which $\calm_n$ is contained in $\Sigma_1$ is given in Remark \ref{Rem B1}). 

Now give a brief sketch of the contents of the article. 
In Section \ref{section 2} we study some properties of pairs $([\cale_1],[\cale_2])$ of mathematical instanton bundles  
and prove the vanishing of certain cohomology groups of their
twists by line bundles $\op3(a)$ and $\op3(-a)$ (see Propositon \ref{Prop 1}). The direct sum $\mathbb{E}=
\cale_1\oplus\cale_2$ is then used in Section \ref{section 3}
as a test rank 4 symplectic instanton bundle. This bundle and
its deformations are used as middle terms of anti-self-dual monads of the form $0\to\op3(-a)\to\mathbb{E}\to\op3(a)\to0$,
the cohomology bundles of which provide general bundles of
the components $\calm_n$ of of the series $\Sigma_0$ (see Theorem \ref{Thm A}). In Section \ref{section 4} we study
direct sums $\mathbb{E}=\cale_1\oplus\cale_2$ of vector bundles, $\cale_i$ are the bundles from the R.~Hartshorne series $\{\calb_0(-1,2n)\}_{n\ge1}$ mentioned above. We prove
certain vanishing properties for cohomology of twists of 
$\cale_i$ (see Proposition \ref{Prop 2}). These properties are then used in Theorem \ref{Thm B} in the construction of general vector bundles of coomponents $\calm_n$ of the series
$\Sigma_1$. In Section \ref{section 5} we give a list of
components $\calm_n\in\Sigma_0$ for $n\le20$ and of
components $\calm_n\in\Sigma_1$ for $n\le40$.

\noindent
\textbf{Conventions and notation}.
\begin{itemize}
\item Everywhere in this paper we work over the base field  of complex numbers $\mathbf{k}=\mathbb{C}$.
\item $\mathbb{P}^3$ is a projective 3-space over $\mathbf{k}$.
\item For a stable rank 2 vector bundle $E$ with $c_1(E)=e$, $c_2(E)=n$ on $\p3$, we denote by $[E]$ its isomorphism class in $\calb(e,n)$. 
\end{itemize}

\noindent
\textbf{Acknowledgements}.
A.~Tikhomirov was supported by the Academic Fund Program at the National Research University Higher School of Economics (HSE) in 2018-2019 
(grant no. 18-01-0037) and by the Russian Academic Excellence Project "5-100". He also acknowledges the hospitality of the Max Planck Institute for Mathematics in Bonn, where this work was partially done during the winter of 2017.

\vspace{5mm}

\section{Some properties of mathematical instantons}
\label{section 2}
 
\vspace{2mm}
Let $a$ and $m$ be two positive integers, where $a\ge2$, and
let $\varepsilon\in\{0,1\}$. In this section we prove the following proposition about mathematical instanton vector bundles which will be used in the proof of Theorem 
\ref{Thm A}.
\begin{proposition}\label{Prop 1}
A general pair
\begin{equation}\label{2 inst}
([\cale_1],[\cale_2])\in I_m\times I_{m+\varepsilon},
\end{equation}
of instanton vector bundles satisfies the following conditions:
\begin{equation}\label{not equal}
[\cale_1]\ne[\cale_2];
\end{equation}
for $i=1,\ m\le a+1,$ respectively, $i=2,\ m+\varepsilon\le a+1$, 
\begin{equation}\label{vanish 1}
h^1(\cale_i(a))=0,
\end{equation}
\begin{equation}\label{vanish 2}
h^2(\cale_i(-a))=0,\ \ \ if\ \ \ a\ge12;
\end{equation}
for $i=1,\ m\le a-4,\ a\ge5,$ respectively, $i=2,\ m+\varepsilon\le a-4,\ a\ge5,$
\begin{equation}\label{vanish 2a}
h^2(\cale_i(-a))=0;
\end{equation}
for $j\ne1$,
\begin{equation}\label{vanish 3}
h^j(\cale_1\otimes\cale_2)=0.
\end{equation}
\end{proposition}

\begin{proof} It is clearly enough to treat the case $i=1$, as the case $i=2$ is treated completely similarly.
Consider two instanton vector bundles such that
the condition (\ref{not equal}) can be evidently achieved. Show that the condition (\ref{vanish 1}) can also be satisfied for general bundles $[\cale_1]\in I_m$ and $[\cale_2]\in I_{m+\varepsilon}$. For this, consider a smooth
quadric surface $S\in\p3$, together with an isomorphism
$S\simeq\p1\times\p1$, and let 
$Y=\overset{m+1}{\underset{i=1}{\sqcup}}l_i$ be a union of $m+1$ distinct projective lines $l_i$ in $\p3$ belonging to one of the two rulings on $S$. Considering $Y$ as a reduced scheme, we have 
$\cali_{Y,S}\simeq\op1(-m-1)\boxtimes\op1$. Thus the exact triple $0\to\cali_{S,\p3}\to\cali_{Y,\p3}\to\cali_{Y,S}\to0$ can be rewritten as
\begin{equation}\label{triple of ideals}
0\to\op3(-2)\to\cali_{Y,\p3}\to\op1(-m-1)\boxtimes\op1\to0.
\end{equation}
Tensor multiplication of (\ref*{triple of ideals}) by $\op3(a+1)$, respectively, by $\op3(a-3)$ 
yields an exact triple
\begin{equation}\label{triple(a)}
0\to\op3(a-1)\to\cali_{Y,\p3}(a+1)\to\op1(a-m)\boxtimes\op1
(a+1)\to0.
\end{equation}
\begin{equation}\label{triple(-a)}
0\to\op3(a-5)\to\cali_{Y,\p3}(a-3)\to\op1(a-4-m)\boxtimes
\op1(a-3)\to0.
\end{equation}
By the K\"unneth formula $h^1(\op1(a-m)\boxtimes\op1(a+1))=0$ for $a\ge2$ and $m\le a+1$, and (\ref{triple(a)}) implies that
\begin{equation}\label{vanish 4}
h^1(\cali_{Y,\p3}(a+1))=0.
\end{equation}
Now consider an extension of $\op3$-sheaves of the form
\begin{equation}\label{trC}
0\to\op3(-1)\to\cale_1\to\cali_{Y,\p3}(1)\to0.
\end{equation}
Such extensions are classified by the vector space $V=\mathrm{Ext}^1(\cali_{Y,\p3}(1),\op3(-1))$, and it is
known that, for a general point $\xi\in V$ the extension sheaf
$\cale_1$ in (\ref{trC}) is a locally free  instanton sheaf  from $I_m$(see, e. g., \cite{NT}) called a {\it 't Hooft instanton}. Now tensoring the triple \eqref{trC} with $\op3(a)$ and passing to cohomology, in view of 
\eqref{vanish 4} we obtain (\ref{vanish 1}) for $i=1$.

To prove \eqref{vanish 2a}, assume that $m\le a-4$; then
similar to \eqref{vanish 4} we obtain using \eqref{triple(-a)}:
\begin{equation}\label{vanish 5}
h^1(\cali_{Y,\p3}(a-3))=0,\ \ \ m\le a-4.
\end{equation}
Tensoring \eqref{trC} with $\op3(a-4)$ we obtain the triple 
\begin{equation}\label{trC new}
0\to\op3(a-5)\to\cale_1(a-4)\xrightarrow{\varepsilon}\cali_{Y,\p3}(a-3)\to0.
\end{equation}
From \eqref{vanish 5} and \eqref{trC new} it follows that
\begin{equation}\label{vanish 6}
h^1(\cale_1(a-4))=0,\ \ \ m\le a-4.
\end{equation}
This together with Serre duality for $\cale_1$ yields \eqref{vanish 2a} for $i=1$.

To prove \eqref{vanish 2}, consider a smooth quadric surface $S'\subset\p3$, together with an isomorphism $S'\simeq\p1\times\p1$, and let $Z=\overset{d}{\underset{i=1}{\sqcup}}\tilde{l}_i$ be a union of $d$ distinct projective lines $\tilde{l}_i$ in $\p3$, belonging to one of the two rulings on $S'$, where $1\le d\le 5$. Considering $Z$ as a reduced scheme, we have 
$\cali_{Z,S'}\simeq\op1(-d)\boxtimes\op1$.
Without loss of generality we may assume that $Z\cap Y=
\emptyset$ and that $Z$ intersects the quadric surface $S$ treated above in $2d$ distinct points $x_1,...,x_{2d}$ such that the points $\mathrm{pr}_2(x_i),\ i=1,...,2d,$ are also distinct, where $\mathrm{pr}_2: S\simeq\p1\times\p1\to\p1$ is the projection onto the second factor. Tensoring the exact triple \eqref{triple of ideals} with $\op3(a-3)$ and restricting it onto $Z$ we obtain a commutative diagram of exact triples 
\begin{equation}\label{diagr 1}
\xymatrix{
& 0 & 0 & 0 &\\
0\ar[r] &\calo_Z(a-5)\ar[r]\ar[u] & 
\calo_Z(a-3)\ar[r]\ar[u] & \underset{1}{\overset{2d}{\oplus}}~\mathbf{k}_{x_j}\ar[u]\ar[r] & 0\\
0\ar[r]& \op3(a-5)\ar[r]\ar^-f[u] & \cali_{Y,\p3}(a-3)\ar[r]
\ar_-{g}[u] &\op1(a-m-4)\boxtimes\op1(a-3)
\ar^-h[u]\ar[r] & 0,}
\end{equation}
where $f,\ g$ and $h$ are the restriction maps.
The sheaf $\ker f=\cali_{Z,\p3}(a-5)$ similar to \eqref{triple(a)} satisfies the exact triple
\begin{equation*}\label{triple(b)}
0\to\op3(a-5)\to\ker f\to\op1(a-d-5)\boxtimes\op1(a-5)\to0.
\end{equation*}
Passing to cohomology of this triple we obtain in view of the conditions $1\le d\le 5$ and $a\ge12$ that $h^1(\ker f)=0$,
i. e. 
$$
h^0(f):\ H^0(\op3(a-5))\to H^0(\calo_Z(a-5))
$$
is an epimorphism. On the other hand, we have: i) $a-m-4\ge0$, 
ii) $a-3\ge2d-1$, since $a\ge12$ and $d\le5$, and iii) the points $\mathrm{pr}_2(x_i),\ i=1,...,2d,$ are distinct. Therefore,
$$
h^0(h):\ H^0(\op1(a-m-4)\boxtimes\op1(a-3))\to H^0(\overset{2d}{\underset{1}{\oplus}}~\mathbf{k}_{x_j})
$$
is also an epimorphism. Whence by the diagram \eqref{diagr 1}
we obtain an epimorphism
\begin{equation}\label{h0(g) epi}
h^0(g):\ H^0(\cali_{Y,\p3}(a-3))\twoheadrightarrow  H^0(\calo_Z(a-3)).
\end{equation}
Now consider the 
$g\circ\varepsilon:\ \cale_1(a-4)\twoheadrightarrow\calo_Z(a-3)$, where
$\varepsilon$ is the epimorphism in the triple \eqref{trC new} and set $E:=\ker(g\circ\varepsilon)\otimes\op3(4-a)$.
Thus, since $\calo_Z=\overset{d}{\oplus}\calo_{\tilde{l}_i}$, we have an exact triple:
\begin{equation}\label{triple with l's}
0\to E(a-4)\to\cale_1(a-4)\xrightarrow{g\circ\varepsilon}
\overset{d}{\underset{1}{\oplus}}\calo_{\tilde{l}_i}(a-3)\to0.
\end{equation}
From the triple \eqref{trC new} it follows that $h^0(\varepsilon):H^0(\cale_1(a-4))\to H^0(\cali_{Y,\p3}(a-3))$ is an epimorphism, hence by \eqref{h0(g) epi} $h^0(g\circ\varepsilon):H^0(\cale_1(a-4))
\to H^0(\overset{d}{\oplus}\calo_{\tilde{l}_i}(a-3))$ is also
an epimorphism. This together with (\ref{triple with l's}) and
\eqref{vanish 6} yields that
\begin{equation}\label{vanish 7}
h^1(E(a-4))=0.
\end{equation}
Note that from \eqref{triple with l's} it follows also that
\begin{equation}\label{c2(E)}
c_2(E)=c_2(\cale_1)+d=m+d\le a+1,
\end{equation}
since $d\le5$ and $m\le a-4$.

Now show that 
\begin{equation}\label{E in closure}
[E]\in\bar{I}_{m+d},
\end{equation}
where $\bar{I}_{m+d}$ is the closure of $I_{m+d}$ in the Gieseker-Maruyama moduli scheme $M(0,m+d,0)$ of semistable rank 2 coherent sheaves with Chern classes $c_1=c_3=0$ and
$c_2=m+d$. (Recall that $M(0,m+d,0)$ is a projective scheme
containing $\calb(0,m+1)$ as an open subscheme - see e. g., \cite{H-vb}, \cite{HL}.) It is enough to treat the case
$d=2$, since the argument for any $d\le5$ is completely similar. Consider the triple \eqref{triple with l's} and denote by $E'_0$ the kernel of the composition 
$$
\cale_1\xrightarrow{g\circ\varepsilon}
\calo_{\tilde{l}_1}(1)\oplus\calo_{\tilde{l}_2}(1)
\xrightarrow{\mathrm{pr_1}}\calo_{\tilde{l}_1}(1).
$$
We then obtain an exact triple
\begin{equation}\label{tr E E'0}
0\to E\to E'_0\xrightarrow{\varepsilon'}\calo_{\tilde{l}_2}(1)
\to0.
\end{equation}
Now we invoke one of the main results of the paper \cite{JMT}
according to which the sheaf $E'_0$ lies in the closure $\bar{I}_{m+1}$ of $I_{m+1}$ in the Gieseker-Maruyama moduli scheme $M(0,m+1,0)$. This implies that there exists a punctured curve $(C,0)\in\bar{I}_{m+1}$ and a flat over $C$ coherent $\calo_{\p3\times C}$-sheaf $\mathbb{E'}$ such that the sheaf $E'_t:=\mathbb{E'}|_{\p3\times\{t\}}$ is an instanton bundle from $I_{m+1}$ for $t\ne0$ and coincides with $E'_0$ for $t=0$. Now, without loss of generality, after possible shrinking the curve $C$, one can extend the epimorphism $\varepsilon'$ in \eqref{tr E E'0} to an epimorphism
$$
\mathbf{e}:\mathbb{E'}\twoheadrightarrow\calo_{\tilde{l}_2}(1)
\boxtimes\calo_C
$$ 
such that $\mathbf{e}\otimes\mathbf{k}(0)=\varepsilon'$.
Set $\mathbb{E}=\ker\mathbf{e}$ and denote $E_t=\mathbb{E}|_{\p3\times\{t\}}$, $t\in C$. As for $t\ne0$ the sheaf $E'_t$ is an instanton bundle from $I_{m+1}$, and 
it fits in an exact triple $0\to E_t\to E'_t\to\calo_{\tilde{l}_2}(1)\to0$, the above mentioned result
from \cite{JMT} yields that $[E_t]\in\bar{I}_{m+2}$ for $t\ne0$. Hence, since $[E_t]\in\bar{I}_{m+2}$ is projective,
it follows that $E_0\in\bar{I}_{m+2}$. Now by construction
$E_0\simeq E$. Thus, $[E]\in\bar{I}_{m+2}$, i. e. we obtain 
the desired result \eqref{E in closure} for $d=2$.
Formula \eqref{vanish 2} now follows from \eqref{vanish 7} for a general $\cale$ by Semicontinuity and Serre duality.

To prove the vanishing \eqref{vanish 3}, consider the triple
\eqref{trC} twisted by $\cale_2$:
\begin{equation}\label{trC times E2}
0\to\cale_2(-1)\to\cale_1\otimes\cale_2
\to\cale_2\otimes\cali_{Y,\p3}(1)\to0,
\end{equation}
and the exact triple
$0\to\cale_2\otimes\cali_{Y,\p3}(1)\to\cale_2(1)\to
\oplus_{i=1}^{m+1}(\cale_2|_{l_i})\to0.$ 
Since $\cale_2$ is an instanton bundle, it follows that
\begin{equation}\label{vanish inst}
h^2(\cale_2(1))=0,\ \ \ h^2(\cale_2(-1))=0.
\end{equation}
On the other hand, without loss of generality, by the Grauert-M\"ulich Theorem \cite[Ch. 2]{OSS} we may assume 
that $\cale_2|_{l_i}\simeq\op1$. This together with the last exact triple and the first equality \eqref{vanish inst} yields
$h^2(\cale_2\otimes\cali_{Y,\p3}(1))=0$. Therefore, in view of \eqref{trC times E2} and the second equality \eqref{vanish inst} we obtain the equality \eqref{vanish 3} for $j=1$.
Last, this equality for $j=0,3$ follows from \eqref{not equal} and the stability of $\cale_1$ and $\cale_2$.
\end{proof}

\begin{remark}\label{rem A}
Note that, under the conditions of Proposition \ref{Prop 1}, the equalities (\ref{vanish 3}) together with Riemann-Roch yield
\begin{equation}\label{h1 12}
h^1(\cale_1\otimes\cale_2)=8m+4\varepsilon-4.
\end{equation}
\end{remark}

\vspace{5mm}

\section{Construction of stable rank two bundles with even determinant}
\label{section 3}

\vspace{5mm}
We first recall the notion of symplectic instanton.
By a \textit{symplectic structure} on a vector bundle $E$ on a scheme $X$ we mean an anti-self-dual isomorphism $\theta: E \xrightarrow{\sim}E^{\vee},\ \theta^{\vee}=-\theta$, considered modulo proportionality. Clearly,
a symplectic vector bundle $E$ has even rank:
\begin{equation*}\label{rk E}
\rk E=2r,\ \ \ r\ge1,
\end{equation*}
and, if $X=\p3$, vanishing odd Chern classes:
\begin{equation*}\label{odd c_i}
c_1(E)=c_3(E)=0.
\end{equation*}
Following \cite{AJTT}, we call a symplectic vector bundle $E$ on $\p3$ a \textit{symplectic instanton} if
\begin{equation}\label{def of sympl inst}
h^0(E(-1))=h^1(E(-2))=h^2(E(-2))=h^3(E(-3))=0,
\end{equation}
and
\begin{equation*}\label{c_2=n}
c_2(E)=n,\ \ \ n\ge1.
\end{equation*}
Consider the instanton bundles $\cale_1$ and $\cale_2$ introduced in Section \ref{section 2} (see Proposition \ref{Prop 1}).
Since $\det\cale_1\simeq\det\cale_2\simeq\op3$, there are symplectic structures $\theta_i:\ \cale_i\xrightarrow{\simeq}\cale_i^{\vee},\ i=1,2,$ which yield a symplectic structure on the direct sum
$\mathbb{E}=\cale_1\oplus\cale_2$:
\begin{equation}\label{sympl str}
\theta=\theta_1\oplus\theta_2:\ \mathbb{E}=\cale_1\oplus\cale_2 \xrightarrow{\simeq}\cale_1^{\vee}\oplus\cale_2^{\vee}=
\mathbb{E}^{\vee}.
\end{equation}
Clearly, $\mathbb{E}$ is a symplectic instanton.

Now assume that $\cale_1$ and $\cale_2$ are chosen in such a way that there exist sections
\begin{equation}\label{empty intersectn}
s_i:\ \op3\to\cale_i(a), \ \ \ such\ that\ \ \ \dim(s_i)_0=1,\ \ i=1,2,\ \ \ (s_1)_0\cap(s_2)_0=\emptyset.
\end{equation}
(Such $\cale_1\in I_m,\ [\cale_2]\in I_{m+\varepsilon}$ always exist; for instance, two general 't Hooft instantons $\cale_1$ and $\cale_2$ satisfy the property (\ref{empty intersectn}) already for $a=1$, hence also for $a\ge2$.) The condition (\ref{empty intersectn}) implies that
the section
\begin{equation}\label{subbundle}
s=(s_1,s_2):\ \op3(-a)\to\mathbb{E}
\end{equation}
is a subbundle morphism, hence its transpose
\begin{equation*}\label{ta}
{}^ts:=s^{\vee}\circ\theta:\ \mathbb{E}\to\op3(a)
\end{equation*}
is a surjection. As $\theta$ in (\ref{sympl str}) is symplectic,
the composition ${}^ts\circ s:\op3(-a)\to\op3(a)$ is also symplectic. Since
$\op3(\pm a)$ are line bundles, it follows that ${}^ts\circ s=0$. Therefore the complex
\begin{equation}\label{monad}
K^{\cdot}:\ \ \ 0\to\op3(-a)\xrightarrow{s}\mathbb{E}\xrightarrow{{}^ts}
\op3(a)\to0
\end{equation}
is a monad and its cohomology sheaf
\begin{equation}\label{coho}
E=\frac{\ker({}^ts)}{\mathrm{im}(s)}
\end{equation}
is locally free. Note that, since the instanton bundles $\cale_1$ and $\cale_2$ are stable, they have zero spaces of global sections, hence also $h^0(\mathbb{E})=0$, and
\eqref{monad} and \eqref{coho} yield $h^0(E)=0$, i. e.
$E$ as a rank 2 vector bundle with $c_1=0$ is stable. 
Besides, since  $c_2(\mathbb{E})=c_2(\cale_1)+c_2(\cale_2)=2m+\varepsilon$,
it follows from \eqref{monad} that 
$c_2(E)=2m+\varepsilon +a^2$. Thus,
\begin{equation*}\label{E in B}
[E]\in\calb(0,2m+\varepsilon +a^2),
\end{equation*} 
and the deformation theory yields that, for any irreducible
component $\calm$ of $\calb(0,2m+\varepsilon +a^2)$,
\begin{equation*}\label{def theory ineq}
\dim\calm\ge1-\chi(\mathcal{E}nd~E)
=8(2m+\varepsilon +a^2)-3.
\end{equation*} 

Note that, since $\cale_1$ and $\cale_2$ are instanton bundles, for $a\ge2$ one has $h^1(\cale_i(-a))=
h^j(\cale_i(a))=0,\ i=1,2,\ j=2,3$, hence by \eqref{sympl str}
\begin{equation}\label{h1bbE(-a)=0}
h^1(\mathbb{E}(-a))=0,\ \ \ a\ge2,
\end{equation}
\begin{equation}\label{h2,3bbE(a)=0}
h^j(\mathbb{E}(a))=0,\ \ \ j=2,3,\ \ \ a\ge2.
\end{equation}
Similarly, in view of \eqref{vanish 1}, 
\begin{equation}\label{h1E(a)=0}
h^1(\mathbb{E}(a))=0,\ \ \ m+\varepsilon\le a+1,\ \ \ a\ge2. 
\end{equation}
This together with \eqref{h2,3bbE(a)=0} and Riemann-Roch yields:
\begin{equation}\label{h0bbE(a)=}
h^0(\mathbb{E}(a))=\chi(\mathbb{E}(a))=
4\binom{a+3}{3}-(2m+\varepsilon)(a+2),\ \ \ 
m+\varepsilon\le a+1,\ \ \ a\ge2.
\end{equation}

Next, 
\begin{equation}\label{decomp EndE}
\mathcal{E}nd~\mathbb{E}\simeq\mathbb{E}\otimes
\mathbb{E}\simeq S^2\mathbb{E}\oplus\wedge^2\mathbb{E}, 
\end{equation}
and it follows from (\ref{sympl str}) that
\begin{equation}\label{decomp S^2E}
S^2\mathbb{E}\simeq S^2\mathcal{E}_1\oplus(\mathcal{E}_1\otimes\mathcal{E}_2)\oplus S^2\mathcal{E}_2,\ \ \ 
\wedge^2\mathbb{E}\simeq \wedge^2\mathcal{E}_1\oplus(\mathcal{E}_1\otimes\mathcal{E}_2)\oplus \wedge^2\mathcal{E}_2.
\end{equation}
Now, since $\mathcal{E}nd~\mathcal{E}_i\simeq\cale_i\otimes
\cale_i\simeq S^2\cale_i\oplus \wedge^2\cale_i,\ \wedge^2\cale_i\simeq\op3,\ i=1,2,$ it follows from \cite{JV} that
$h^1(\mathcal{E}nd~\mathcal{E}_1)\simeq h^1(S^2\cale_1)=8m-3,\ h^1(\mathcal{E}nd~\mathcal{E}_2)\simeq h^1(S^2\cale_2)=8m+8\varepsilon-3,$ and 
$h^j(\mathcal{E}nd~\mathcal{E}_i)=h^j(S^2\cale_i)=0,\ i=1,2,\ j\ge2$. This together with \eqref{decomp EndE}-(\ref{decomp S^2E}), (\ref{vanish 1}) and (\ref{h1 12}) implies that
\begin{equation}\label{h1 S2bbE}
h^1(\mathcal{E}nd~\mathbb{E})=32m+16\varepsilon-14,\ \ \ 
h^1(S^2\mathbb{E})=24m+12\varepsilon-10,
\end{equation}
\begin{equation}\label{hi S2bbE}
h^i(\mathcal{E}nd~\mathbb{E})=h^i(S^2\mathbb{E})=0,\ \ \ i\ge2.
\end{equation}

Now assume that 
\begin{equation}\label{basic cond}
either\ \ \ 5\le a\le 12,\ 1+\varepsilon\le m+\varepsilon\le a-4,\ \ \ or\ \ \ 
a\ge12,\ 1+\varepsilon\le m+\varepsilon\le a+1.
\end{equation} 
It follows from (\ref{vanish 2}), (\ref{vanish 2a}) and
\eqref{sympl str} that
\begin{equation}\label{h2E(-a)=0}
h^2(\mathbb{E}(-a))=0.
\end{equation}
Consider the total complex $T^{\cdot}$ of the double komplex $K^{\cdot}\otimes K^{\cdot}$, where $K^{\cdot}$ is the monad (\ref{monad}):
\begin{equation*}\label{total T}
\begin{split}
& T^{\cdot}:\ \ \ 0\to\op3(-2a)\xrightarrow{d_{-2}}
2\mathbb{E}(-a)\xrightarrow{d_{-1}}\mathbb{E}\otimes
\mathbb{E}\oplus2\op3\xrightarrow{d_0}
2\mathbb{E}(a)\xrightarrow{d_1}\op3(2a)\to0,\\
& \ \ \ \ \ \ \ \ \ \ \ \ \ \ \ \ \ \ \ \ \ \ \ \ \ \ \ \ \ \ \ \ \ \ \ \ \ \ \ \ \ \  E\otimes E=\frac{\ker(d_0)}
{\mathrm{im}(d_{-1})}.
\end{split}
\end{equation*}
Following Le Potier \cite{LP}, consider the
symmetric part $ST^{\cdot}$ of the complex $T^{\cdot}$:
\begin{equation*}\label{sym of monad}
ST^{\cdot}:\ \ \ 0\to\mathbb{E}(-a)\xrightarrow{\alpha}S^2\mathbb{E}\oplus
\op3\xrightarrow{{}^t\alpha}\mathbb{E}(a)\to0,\ \ \ \ \ \ \ 
S^2E=\frac{\ker({}^t\alpha)}{\mathrm{im}(\alpha)},
\end{equation*}
where $\alpha$ is the induced subbundle map.
The inclusion of complexes $ST^{\cdot}\hookrightarrow T^{\cdot}$
induces commutative diagrams
\begin{equation}\label{diagr D1}
\xymatrix{
0\ar[r] & \mathbb{E}(-a))\ar[r]\ar@{_{(}->}[d] & 
\ker({}^t\alpha)\ar[r]\ar@{_{(}->}[d] & S^2E
\ar[r]\ar@{_{(}->}[d] & 0\\
0\ar[r] & \mathrm{im}(d_{-1})\ar[r]& \ker(d_0) \ar[r] & E\otimes E\ar[r] & 0,}
\end{equation}
\begin{equation}\label{diagr D2}
\xymatrix{
0\ar[r] & \ker({}^t\alpha)\ar[r]\ar@{_{(}->}[d] & 
S^2\mathbb{E}\oplus\op3\ar[r]\ar@{_{(}->}[d] & \mathbb{E}(a)
\ar[r]\ar@{_{(}->}[d] & 0\\
0\ar[r] & \ker(d_0)\ar[r]& \mathbb{E}\otimes\mathbb{E}\oplus
2\op3 \ar[r] & \mathrm{im}(d_0)\ar[r] & 0,}
\end{equation}
and an exact triple
\begin{equation}\label{exact ker}
0\to\op3(-2a)\xrightarrow{d_{-2}}2\mathbb{E}(-a)
\to\mathrm{im}(d_{-1})\to0
\end{equation}
Passing to cohomology in \eqref{diagr D1}-\eqref{exact ker} and
using \eqref{h1bbE(-a)=0}, \eqref{h2E(-a)=0}, \eqref{h1E(a)=0}
and the equality $h^0(S^2\mathbb{E})=0$, we obtain an equality $h^0(\mathrm{coker}\alpha)=1$ and an exact sequence
\begin{equation}\label{seq S2E}
0\to H^0(\mathbb{E}(a))/\mathbb{C}\to
H^1(S^2E)\xrightarrow{\mu}H^1(S^2\mathbb{E})\to0,
\end{equation}
which fits in a commutative diagram
\begin{equation}\label{diagr D1}
\xymatrix{
0\ar[r] & H^0(\mathbb{E}(a))/\mathbb{C})\ar[r] & 
H^1(S^2E)\ar[r]^-\mu\ar@{=}[d] & H^1(S^2\mathbb{E})
\ar[r]\ar@{_{(}->}[d] & 0\\
& & H^1(E\otimes E)) \ar[r] & \mathbb{E}\otimes 
\mathbb{E}. &}
\end{equation}

From \eqref{h0bbE(a)=}, \eqref{h1 S2bbE} and \eqref{h1 S2E} it
follows that
\begin{equation}\label{h1 S2E}
h^1(S^2E)=h^0(\mathbb{E}(a))+24m+12\varepsilon-11=
4\binom{a+3}{3}+(2m+\varepsilon)(10-a)-11.
\end{equation}

Note that, since $E$ is a stable rank-2 bundle, $H^1(\cale nd~E)=H^1(S^2E)$ is isomorphic to the Zariski tangent space $T_{[E]}\calb(0,2m+\varepsilon +a^2)$:
\begin{equation}\label{Kod-Sp1}
\theta_{E}:\ T_{[E]}\calb(0,2m+\varepsilon +a^2)\simto
H^1(\cale nd~E)=H^1(S^2E). 
\end{equation}
(Here $\theta_{E}$ is the Kodaira-Spencer isomorphism.)
Thus, we can rewrite \eqref{h1 S2E} as
\begin{equation}\label{dim Zar tang sp}
\dim T_{[E]}\calb(0,2m+\varepsilon +a^2)=
4\binom{a+3}{3}+(2m+\varepsilon)(10-a)-11.
\end{equation}

We will now prove the following main result of this section.
\begin{theorem}\label{Thm A}
Under the condition \eqref{basic cond}, there exists an irreducible family $\calm_n(E)\subset\calb(0,n)$, where $n=2m+\varepsilon +a^2$, of dimension given by the right hand side of \eqref{dim Zar tang sp} and containing the above constructed
point $[E]$. Hence the closure $\calm_n$ of $\calm_n(E)$ in
$\calb(0,n)$ is an irreducible component of $\calb(0,n)$.
The set $\Sigma_0$ of these components $\calm_n$ is an infinite series distinct from the series of instanton components $\{I_n\}_{n\ge1}$ and from the series of components described in \cite{Ein} and \cite{AJTT}. 
Furthermore, at least for each $n\ge146$ there exists an irreducible component $\calm_n$ of $\calb(0,n)$ belonging to the series $\Sigma_0$.
\end{theorem}

\begin{proof}
According to J. Bingener \cite[Appendix]{BH}, the equality $h^2(\cale nd\mathbb{E})=0$ (see \eqref{hi S2bbE})
implies that there exists (over $\mathbf{k}=\mathbb{C}$)
a versal deformation of the bundle $\mathbb{E}$, i. e. a smooth variety $B$ of dimension $\dim B=h^1(\cale nd\mathbb{E})$, with a marked point $0\in B$, and a locally free sheaf  $\boldsymbol{\cale}$ on $\p3\times B$ such that $\boldsymbol{\cale}|_{\p3\times\{0\}}\simeq\mathbb{E}$ and the
Koaira-Spencer map $\theta:\ T_{[\mathbb{E}]}B\to 
H^1(\cale nd\mathbb{E})$ is an isomorphism. For $b\in B$ denote $E_b:=\boldsymbol{\cale}|_{\p3\times\{b\}}$ and consider in $B$ a closed subset
\begin{equation*}
U=\{b\in B\ |\ E_b\ is\ a\ symplectic\ instanton\}.
\end{equation*}
By definition, $U=\tilde{U}\cap B^*$, where 
$\tilde{U}=\{b\in B\ |\ E_b\ is\ a\ symplectic\ bundle\}$ is a closed subset of $B$ and 
\begin{equation*}
\begin{split}
&B^*=\{b\in B\ |\ E_b \ satisfies\ the\ vanishing\ conditions\ \eqref{def of sympl inst}\ and\ the\ condition\ \\ 
&  h^0(E_b)=h^i(E_b(-a))=h^j(E_b(a))=h^k(S^2E_b)=0,\ i=1,2,\ j\ge1,\ k=0,2,3\}
\end{split} 
\end{equation*}
is an open subset of $B$ by the Semicontinuity. (Here $a$ is taken from \eqref{basic cond}). Since $\mathbb{E}$ is symplectic, so that
$\cale nd\mathbb{E}\simeq S^2\mathbb{E}\oplus\wedge^2\mathbb{E}$, 
it follows from \cite{R} that the Kodaira-Spencer map $\theta$ yields an isomorphism $\theta:T_{[\mathbb{E}]}U=T_{[\mathbb{E}]}\tilde{U}\simto H^1(S^2\mathbb{E})$. 
Thus, $U$ is a smooth variety of dimension  
$$
\dim U=h^1(S^2\mathbb{E})=24m+12\varepsilon−10.
$$
(We use Riemann-Roch and the vanishing of $h^i(S^2\mathbb{E}),\ i\ne1$.) 

Let $p:\p3\times B\to B$
be the projection. By the Base-Change and the vanishing conditions defining $B^*$, respectively, $U$ the sheaf $\cala:=p_*(\boldsymbol{\cale}\otimes\op3(a)\boxtimes\calo_U)$ is a locally free sheaf of rank $\chi(\mathbb{E}(a))=h^0(\mathbb{E}(a))$ given by \eqref{h0bbE(a)=}. Hence $\pi:\ \tilde{X}=\mathbf{Proj}
(S^{\cdot}_{\op3}\cala^{\vee})\to U$ is a projective bundle
with the Grothendieck sheaf $\calo_{\tilde{X}/U}(1)$ and a
morphism 
$\mathbf{s}:\ \op3\boxtimes\calo_{\tilde{X}/U}(-1)\to
\tilde{\pi}^*(\boldsymbol{\cale}\otimes\op3(a)
\boxtimes\calo_U)$
defined as the composition of canonical evaluation morphisms
$\op3\boxtimes\calo_{\tilde{X}/U}(-1)\to\tilde{p}^*\pi^*\cala
\to\tilde{\pi}^*(\boldsymbol{\cale}\otimes\op3(a)\boxtimes
\calo_U)$, where $\tilde{X}\xleftarrow{\tilde{p}}\p3\times
\tilde{X}\xrightarrow{\tilde{\pi}}\p3\times U$ are the induced
projections.

Let $X=\{x\in\tilde{X}\ |\ \mathbf{s}^{\vee}|_{\p3\times\{x\}}\ $is surjective $\}$. This is an open dense subset of the smooth irreducible variety $\tilde{X}$ since it contains the point $x_0=(s:\ \op3\to\mathbb{E}(a))$ given in \eqref{subbundle}.
Hence $X$ is smooth and irreducible. In addition, since $\boldsymbol{\cale}$ is a versal family of bundles, it follows
that $X$ is an open subset of the Quot-scheme 
$\mathrm{Quot}_{\p3\times B/B}(\boldsymbol{\cale},P(n))$, where 
$P(n):=\chi(\op3(a+n))$. Therefore, by \cite[Prop. 2.2.7]{HL}
in view of \eqref{h1E(a)=0} there is an exact triple
\begin{equation}\label{tangent seq}
0\to H^0(\mathbb{E}(a))/\mathbb{C}\to T_{x_0}X\xrightarrow{d\pi}T_{[\mathbb{E}]}B\to0,
\end{equation}
which is obtained as the cohomology sequence
\begin{equation}\label{coho seq}
0\to H^0(\mathbb{E}(a))/\mathbb{C}\to H^1(\calh om(F,\mathbb{E}))\to H^1(\cale nd~\mathbb{E})\to0
\end{equation}
of the exact triple 
$0\to\calh om(F,\mathbb{E})\to\calh om(\mathbb{E},
\mathbb{E})\to\mathbb{E}(a)\to0$ obtained by applying the functor $\calh om(-,\mathbb{E})$ to the exact triple
$0\to\op3(-a)\xrightarrow{s}\mathbb{E}\to F\to0$, 
where $F:=\mathrm{coker}(s)$. 

Next, since $\boldsymbol{\cale}$ is a versal family of bundles, it follows that $\mathbf{E}=\boldsymbol{\cale}|_{\p3\times U}$ is a versal family of symplectic instantons. Hence, denoting 
$Y=U\times_BX$, we extend the exact triple \eqref{tangent seq}
to commutative diagram
\begin{equation}\label{diagr 4}
\xymatrix{
0\ar[r] & H^0(\mathbb{E}(a))/\mathbb{C}\ar[r] & 
T_{x_0}X\ar[r] & T_{[\mathbb{E}]}B\ar[r] & 0\\
0\ar[r] & H^0(\mathbb{E}(a))/\mathbb{C}\ar[r]\ar[r]\ar@{=}[u] & T_{x_0}Y \ar[r]^-{d\pi}
\ar@{^{(}->}^-{i_Y}[u] & T_{[\mathbb{E}]}U
\ar@{^{(}->}^-{i_U}[u]\ar[r] & 0,}
\end{equation}
where $i_Y$ and $i_U$ are natural inclusions. (Note that, under the Kodaira-Spencer isomorphisms $\theta:T_{[\mathbb{E}]}U)
\xrightarrow{\simeq}H^1(S^2\mathbb{E})$ and $T_{[\mathbb{E}]}B
\xrightarrow{\simeq}H^1(\mathbb{E}\otimes\mathbb{E})\simeq
H^1(\cale nd~\mathbb{E})$ the rightmost inclusions in diagrams
\eqref{diagr D1} and \eqref{diagr 4} coincide.)
Consider the modular morphism
$$
\Phi:\ Y\to\calb:=\calb(0,2m+\varepsilon +a^2),\ \ 
(b,s)\mapsto\bigg[\frac{\mathrm{Ker}
({}^ts)}{\mathrm{Im}(s)}\bigg],
$$
where, as before, $s:\op3(-a)\to E_b$ is a subbundle morphism.
Its differential $d\Phi$ composed with the Kodaira-Spencer map
$\theta_E$ from \eqref{Kod-Sp1} is a linear map
\begin{equation*}
\phi=\theta_E\circ d\Phi:\ T_{x_0}Y\to H^1(S^2E)=H^1(E\otimes E).
\end{equation*}
Now from the functorial properties of the Kodaira-Spencer maps
$\phi$ and $\theta$ it follows that the triple \eqref{seq S2E} and the lower triple in the diagram \eqref{diagr 4} fit in a
commutative diagram
\begin{equation*}
\xymatrix{
0\ar[r] & H^0(\mathbb{E}(a))/\mathbb{C}\ar[r]\ar[r] & H^1(S^2E) \ar[r]^-{\mu} & H^1(S^2\mathbb{E})\ar[r] & 0\\
0\ar[r] & H^0(\mathbb{E}(a))/\mathbb{C}\ar[r]\ar[r]\ar@{=}[u] & T_{x_0}Y \ar[r]^-{d\pi}\ar^-{\phi}[u] & T_{[\mathbb{E}]}U
\ar^-{\theta}_-{\simeq}[u]\ar[r] & 0.}
\end{equation*}
This implies that $\phi$ is an isomorphism, so that, since $Y$
is smooth at $x_0$ and irreducible, $\calm_n(E)=\Phi(Y)$ is an open subset of an irreducible component $\calm_n$ of $\calb(0,n)$, of dimension given by \eqref{dim Zar tang sp}. 

It is easy to check that the dimension $\dim\calm_n$ given by
\eqref{dim Zar tang sp}, with $m,\varepsilon$ and $a$ subjected to the condition \eqref{basic cond}, satisfies 
the strict inequality $\dim\calm_n>8n-3=\dim I_n$. 
This shows that the series $\Sigma_0$ is distinct from $\{I_n\}_{n\ge1}$. To distinguish $\Sigma_0$ from
from the series of components described in \cite{Ein}, it is enough to see that the spectra of general bundles of these two series are different. (We leave to the reader a direct verification of this fact.)  

Note that, for each $a\ge12$ we have $1\le m\le a+1$ and $0\le\varepsilon\le1$, so that $n=2m+\varepsilon +a^2$
ranges through the whole interval of positive integers $(a^2+2,(a+1)^2+1)\subset\mathbb{Z}_{+}$. Hence, $n$ takes
at least all positive values $\ge12^2+2=146$. This shows that
for each $n\ge146$ there exists an irreducible component $\calm_n\in\Sigma_0$. 

Last, remark that, for the series of components $\calm_n$ described in \cite{AJTT}, $n$ takes values $n=1+k^2,\ k\in\{2\}\cup(4,\infty)$. Hence this series is distinct from $\Sigma_0$.
Theorem is proved.
\end{proof}

\vspace{5mm}

\section{Construction of stable rank two bundles with odd determinant}
\label{section 4}

\vspace{5mm}
In this section we will construct an infinite series of stable vector bundles from $\calb(-1,2m)$,  $m\in\mathbb{Z}_+$. It is known from \cite[Example 4.3.2]{H-vb} that, for each $m\ge1$ there exists an irreducible component $\calb_0(-1,2m)$ of $\calb(-1,2m)$,
of the expected dimension
\begin{equation}\label{dim B0}
\dim\calb_0(-1,2m)=16m-5,
\end{equation}
which contains bundles $\cale_1$ obtained via the Serre constructions as the extensions of the form
\begin{equation}\label{Serre-1}
0\to\op3(-2)\to\cale_1\to\cali_{Y,\p3}(1)\to0,
\end{equation}
where $Y$ is a union of $m+1$ disjoint conics in $\p3$.

Below we will need the following analogue of the Proposition \ref{Prop 1}.
\begin{proposition}\label{Prop 2}
Let $a,m\in\mathbb{Z}_+$, $a\ge2$, and
let $\varepsilon\in\{0,1\}$. A general pair
\begin{equation}\label{2 inst a}
([\cale_1],[\cale_2])\in \calb_0(-1,2m)\times \calb_0(-1,2(m+\varepsilon)),
\end{equation}
of vector bundles satisfies the following conditions:
$$
[\cale_1]\ne[\cale_2];
$$
for $i=1,\ a\ge2m+4,$ respectively, for $i=2,\ a\ge2(m+\varepsilon)+4$, 
\begin{equation}\label{vanish 1a}
h^1(\cale_i(a))=0,
\end{equation}
\begin{equation}\label{vanish 2a a}
h^2(\cale_i(-a))=0;
\end{equation}
\begin{equation}\label{vanish 2b}
h^1(\cale_i(-a))=0;
\end{equation}
\begin{equation}\label{vanish 3a}
h^j(\cale_1(1)\otimes\cale_2)=0,\ \ \ \ j\ne1.
\end{equation}
\end{proposition}

\begin{proof}
Let  $Y=\sqcup_1^{m+1}C_i$ be a disjoint union of conics $C_i=l_i\cup l'_i$ decomposable into pairs of distinct lines $l_i,l'_i$, such that\\
(i) there exist two smooth quadrics $S\simeq\p1\times\p1$
and  $S'\simeq\p1\times\p1$ with the property that
$l_1,...,l_{m+1}$, respectively, $l'_1,...,l'_{m+1}$ are the lines of one ruling on $S$, respectively, on
$S'$; for instance, denoting $Y_0=l_1\sqcup...\sqcup l_{m+1}, \ Y'=l'_1\sqcup...\sqcup l'_{m+1},$
we may assume that
$$
\calo_S(Y_0)\simeq\op1(m+1)\boxtimes\op1,\ \ \ 
\calo_{S'}(Y')\simeq\op1(m+1)\boxtimes\op1;
$$
(ii) the set of $m+1$ distinct points $Z=(Y'\cap S)\setminus(Y_0\cap Y')$ satisfies the condition that
$pr_1(Z)$ is a union of $m+1$ distinct points, where
$pr_1:S'\to\p1$ is the projection. 
We then have a diagram similar to \eqref{diagr 1}:
\begin{equation}\label{diagr 6}
\xymatrix{
& 0 & 0 & 0 &\\
0\ar[r] &\calo_{Y'}(a-4)\ar[r]\ar[u] & 
\calo_{Y'}(a-3)\ar[r]\ar[u] & \calo_{Z}(a-3) \ar[u]\ar[r] & 0\\
0\ar[r]& \op3(a-4)\ar[r]\ar^-f[u] & \cali_{Y_0,\p3}(a-2)\ar[r]
\ar_-{g}[u] &\op1(a-m-3)\boxtimes\op1(a-2)
\ar^-h[u]\ar[r] & 0.}
\end{equation}
Under the assumptions $a\ge2m+4$ and $m\ge2$ the cohomology of the lower triple of this diagram yields
\begin{equation}\label{h1(IY0)}
h^1(\cali_{Y_0,\p3}(a-2))=0.
\end{equation} 
Next, similar to \eqref{triple(a)} we have an exact triple
$0\to\op3(a-6)\to\cali_{Y,\p3}(a-4)\to\op1(a-5-m)\boxtimes\op1
(a-4)\to0$ which implies that $h^1(\cali_{Y,\p3}(a-4))=0$ since $a-5-m\ge0$ for $a\ge2m+4$ and $m\ge1$. Since
$\cali_{Y,\p3}(a-4)=\ker f$, it follows that

\begin{equation}\label{h0(f) is surj}
h^0(f):\ H^0(\op3(a-4))\to H^0(\calo_{Y'}(a-4))\ \ \ is\ surjective.
\end{equation}
On the other hand, since $a-3-m\ge m+1=h^0(Z)$, from the above condition (ii) on $Z$ it follows that 
$h^0(h):\ H^0(\op1(a-m-3)\boxtimes\op1(a-2))\to H^0(\calo_{Z}(a-3))$ is surjective. This together with \eqref{h0(f) is surj} and diagram \eqref{diagr 6} yields
that 
$h^0(g):\ H^0(\cali_{Y_0,\p3}(a-2))\to H^0(\calo_{Y'}(a-3))$
is surjective. Since $\ker g\simeq\cali_{Y,\p3}(a-2)$, it follows by \eqref{h1(IY0)} that
\begin{equation}\label{h1(IY)}
h^1(\cali_{Y,\p3}(a-2))=0.
\end{equation}
Now, twisting the triple \eqref{Serre-1} by $\op3(a-3)$ and
using \eqref{h1(IY)} we obtain $h^1(\cale_1(a-3))=0$, hence by Serre duality $h^2(\cale_1(-a))=0$. Besides, $h^1(\cale_1(a-3))=0$ clearly implies $h^1(\cale_1(a))=0$
since $a\ge2m+4.$ Now, by Semicontinuity, this yields
\eqref{vanish 1a} and \eqref{vanish 2a} for a general $[\cale_1]\in\calb_0(-1,2m)$. The same equalities are clearly true for $i=2$. 

Next, since $a\ge2$, it follows that $h^0(\calo_{C_i}(1-a))=0$ for any conic $C_i\subset Y$, hence the cohomology of the triple $0\to\cali_{Y,\p3}(1-a)\to\op3(1-a)\to
\oplus_{i=1}^{m+1}\calo_{C_i}(1-a)\to0$ yields $h^1(\cali_{Y,\p3}(1-a))=0$; this together with \eqref{Serre-1} and the Semicontinuity yields \eqref{vanish 2b} for $i=1$ and similarly for $i=2$.

Last, the equalities \eqref{vanish 3a} are proved similarly to \eqref{vanish 3}.
\end{proof}

\begin{remark}\label{rem B}
Note that, under the conditions of Proposition \ref{Prop 2}, the equalities (\ref{vanish 3a}) together with Riemann-Roch yield
\begin{equation}\label{h1 12 a}
h^1(\cale_1(1)\otimes\cale_2)=16m+8\varepsilon-6.
\end{equation}
\end{remark}

Now, to construct new series of components of $\calb(-1,4m+2\varepsilon)$, we proceed along the same lines as in Section \ref{section 3}. We first introduce the notion
of a twisted symplectic structure on a vector bundle.
By a \textit{twisted symplectic structure} on a vector bundle $E$ on $\p3$ we mean an isomorphism 
$\theta: E \xrightarrow{\sim}E^{\vee}(-1)$ such that $\theta^{\vee}(1)=-\theta$, considered modulo proportionality. (Here by definition $\theta^{\vee}(1):= \theta^{\vee}\otimes\mathrm{id}_{\op3(1)}$.) Clearly,
a vector bundle $E$ with twisted symplectic structure has even rank:\ \ $\rk E=2r,\ r\ge1.$

Consider the vector bundles $\cale_1$ and $\cale_2$ introduced in Proposition \ref{Prop 2}.
Since $\det\cale_1\simeq\det\cale_2\simeq\op3(-1)$, there are twisted symplectic structures $\theta_i:\ \cale_i\xrightarrow{\simeq}\cale_i^{\vee}(-1),\ i=1,2,$ which yield a twisted symplectic structure on the direct sum
$\mathbb{E}=\cale_1\oplus\cale_2$:
\begin{equation}\label{twisted sympl str}
\theta=\theta_1\oplus\theta_2:\ \mathbb{E}=\cale_1\oplus\cale_2 \xrightarrow{\simeq}\cale_1^{\vee}(-1)\oplus\cale_2^{\vee}(-1)=\mathbb{E}^{\vee}(-1).
\end{equation}

Now assume that $\cale_1$ and $\cale_2$ are chosen in such a way that there exist sections
\begin{equation}\label{empty intersectn a}
s_i:\ \op3\to\cale_i(a+1), \ \ \ s.t.\ \ \ \dim(s_i)_0=1,\ \ i=1,2,\ \ \ (s_1)_0\cap(s_2)_0=\emptyset.
\end{equation}
(Such $\cale_1\in\calb_0(-1,2m),\ [\cale_2]\in \calb_0(-1,2(m+\varepsilon))$ always exist, since already for $a=1$, hence also for $a\ge2$ two general bundles of the form \eqref{Serre-1} satisfy the property (\ref{empty intersectn a}).) The assumption (\ref{empty intersectn a}) implies that
the section $s=(s_1,s_2):\ \op3(-a-1)\to\mathbb{E}$
is a subbundle morphism, hence its transpose
${}^ts:=s^{\vee}(-1)\circ\theta:\ \mathbb{E}\to\op3(a)$
is an epimorphism. As $\theta$ in (\ref{twisted sympl str}) is twisted symplectic, the composition ${}^ts\circ s:\op3(-a-1)\to\op3(a)$ is also twisted symplectic. Therefore, since
$\op3(a)$ and $\op3(-a-1)$ are line bundles, it follows that ${}^ts\circ s=0$, i. e. the complex
\begin{equation}\label{monad a}
K^{\cdot}:\ \ \ 0\to\op3(-a-1)\xrightarrow{s}\mathbb{E}\xrightarrow{{}^ts}
\op3(a)\to0,\ \ \ \ \ E=\frac{\ker({}^ts)}{\mathrm{im}(s)},
\end{equation}
is a monad and its cohomology sheaf $E$ is locally free. Note that, since the bundles $\cale_1$ and $\cale_2$ are stable, they have zero spaces of global sections, hence also $h^0(\mathbb{E})=0$, and \eqref{monad a}
yields $h^0(E)=0$, i. e. $E$ as a rank 2 vector bundle with $c_1=-1$ is stable. Besides, since  $c_2(\mathbb{E})=c_2(\cale_1)+c_2(\cale_2)=4m+2\varepsilon$,
it follows from \eqref{monad a} that 
$c_2(E)=4m+2\varepsilon +a(a+1)$. Thus,
$$
[E]\in\calb(-1,4m+2\varepsilon +a(a+1)),
$$
and the deformation theory yields that, for any irreducible
component $\calm$ of $\calb(-1,4m+2\varepsilon +a(a+1))$,
$$
\dim\calm\ge1-\chi(\mathcal{E}nd~E)=8(4m+2\varepsilon +a(a+1))-5.
$$

Now, as in \eqref{sym of monad}, consider the
symmetric part of the total complex of the double komplex 
$K^{\cdot}\otimes (K^{\cdot})^{\vee}$, where $K^{\cdot}$ is the monad (\ref{monad a}):
\begin{equation}\label{sym of monad a}
0\to\mathbb{E}(-a)\xrightarrow{\alpha}S^2\mathbb{E}(1)
\oplus
\op3\xrightarrow{{}^t\alpha}\mathbb{E}(a+1)\to0,\ \ \ \ 
S^2E(1)=\frac{\ker({}^t\alpha)}{\mathrm{im}(\alpha)}.
\end{equation}
Here $\alpha$ is the induced subbundle map and $S^2E(1)$ is its cohomology sheaf. The monad (\ref{sym of monad a}) can be rewritten as a diagram of exact triples similar to \eqref{diagr 2}:
\begin{equation}\label{diagr 2a}
\xymatrix{
& & & 0 &\\
& & & \mathbb{E}(a+1)\ar[u]\\
 0\ar[r]& \mathbb{E}(-a)\ar^-\alpha[r] & S^2\mathbb{E}(1)\oplus\op3\ar[r]&
\mathrm{coker}\alpha\ar[u]\ar[r] & 0\\
& & & S^2E(1)\ar[u]&\\
& & & 0.\ar[u]& }
\end{equation}
Note that, by \eqref{vanish 2b} and \eqref{twisted sympl str} one has 
\begin{equation}\label{h1bbE(-a)=0 a}
h^1(\mathbb{E}(-a))=0,\ \ \ a\ge2,
\end{equation}
\begin{equation}\label{h2,3bbE(a)=0 a}
h^j(\mathbb{E}(a+1))=0,\ \ \ j=2,3,\ \ \ a\ge2m+3.
\end{equation}
Similarly, in view of \eqref{vanish 1a}, 
\begin{equation}\label{h1E(a)=0 a}
h^1(\mathbb{E}(a+1))=0,\ \ \ a\ge2m+3. 
\end{equation}
This together with \eqref{h2,3bbE(a)=0 a} and Riemann-Roch yields:
\begin{equation}\label{h0bbE(a)= a}
h^0(\mathbb{E}(a+1))=\chi(\mathbb{E}(a+1))=
4\binom{a+3}{3}+2\binom{a+3}{2}-(2m+\varepsilon)(2a+5).
\end{equation}

Next, 
\begin{equation}\label{decomp EndE a}
\mathcal{E}nd~\mathbb{E}\simeq\mathbb{E}(1)\otimes
\mathbb{E}\simeq S^2\mathbb{E}(1)\oplus\wedge^2\mathbb{E}(1), 
\end{equation}
and it follows from (\ref{twisted sympl str}) that
\begin{equation}\label{decomp S^2E a}
\begin{split}
& S^2\mathbb{E}(1)\simeq S^2\mathcal{E}_1(1)\oplus(\mathcal{E}_1(1)\otimes
\mathcal{E}_2)\oplus S^2\mathcal{E}_2(1),\\
& \wedge^2\mathbb{E}(1)\simeq \wedge^2\mathcal{E}_1(1)\oplus(\mathcal{E}_1(1)\otimes
\mathcal{E}_2)\oplus\wedge^2\mathcal{E}_2(1).
\end{split}
\end{equation}
Now, since $\mathcal{E}nd~\mathcal{E}_i\simeq\cale_i(1)\otimes
\cale_i\simeq S^2\cale_i(1)\oplus \wedge^2\cale_i(1),\ \wedge^2\cale_i\simeq\op3,\ i=1,2,$ it follows from \cite{JV} that
$h^1(\mathcal{E}nd~\mathcal{E}_1)\simeq h^1(S^2\cale_1(1))=16m-5,\ h^1(\mathcal{E}nd~\mathcal{E}_2)\simeq h^1(S^2\cale_2(1))=16(m+\varepsilon)-5,$ and 
$h^j(\mathcal{E}nd~\mathcal{E}_i)=h^j(S^2\cale_i(1))=0,\ i=1,2,\ j\ge2$. This together with \eqref{decomp EndE a}-(\ref{decomp S^2E a}), (\ref{vanish 1a}) and (\ref{h1 12 a}) implies that
\begin{equation}\label{h1 S2bbE a}
h^1(\mathcal{E}nd~\mathbb{E})=64m+32\varepsilon-22,\ \ \ 
h^1(S^2\mathbb{E}(1))=48m+24\varepsilon-16,
\end{equation}
\begin{equation*}\label{hi S2bbE a}
h^i(\mathcal{E}nd~\mathbb{E})=h^i(S^2\mathbb{E}(1))=0,\ \ \ i\ge2.
\end{equation*}
It follows from (\ref{vanish 2a a}) and
\eqref{twisted sympl str} that
\begin{equation}\label{h2E(-a)=0 a}
h^2(\mathbb{E}(-a))=0.
\end{equation}
Note that \eqref{h1bbE(-a)=0 a}, \eqref{h2E(-a)=0 a} and
\eqref{h1E(a)=0 a}, together with the diagram \eqref{diagr 2a} yield an equality $h^0(\mathrm{coker}\alpha)=1$ and an exact sequence:
$$
0\to H^0(\mathbb{E}(a+1))/\mathbb{C}\to
H^1(S^2E(1))\xrightarrow{\mu}H^1(S^2\mathbb{E}(1))\to0,
$$
hence by \eqref{h0bbE(a)= a} and \eqref{h1 S2bbE a} we have
\begin{equation}\label{h1 S2E a}
\begin{split}
& h^1(S^2E(1))=h^0(\mathbb{E}(a+1))+48m+24\varepsilon-17=\\
& 4\binom{a+3}{3}+2\binom{a+3}{2}-(2m+\varepsilon)(2a-19)
-17.
\end{split}
\end{equation}

Note that, since $E$ is a stable rank-2 bundle, $H^1(\cale nd~E)=H^1(S^2E(1))$ is isomorphic to the Zariski tangent space $T_{[E]}\calb(-1,4m+2\varepsilon +a(a+1))$:
\begin{equation}\label{Kod-Sp1 a}
\theta_{E}:\ T_{[E]}\calb(-1,4m+2\varepsilon +a(a+1))\simto
H^1(\cale nd~E)=H^1(S^2E(1)). 
\end{equation}
(Here $\theta_{E}$ is the Kodaira-Spencer isomorphism.)
Thus, 
we can rewrite \eqref{h1 S2E} as
\begin{equation}\label{dim Zar tang sp a}
\begin{split}
& \dim T_{[E]}\calb(-1,4m+2\varepsilon +a(a+1))=\\
& 4\binom{a+3}{3}+2\binom{a+3}{2}-(2m+\varepsilon)(2a-19)
-17.
\end{split}
\end{equation}
\begin{theorem}\label{Thm B}
For $m\ge1$, $\varepsilon\in\{0,1\}$ and $a\ge2(m+\varepsilon)+3$, there exists an irreducible family $\calm_n(E)\subset\calb(-1,n)$, where $n=4m+2\varepsilon +a(a+1)$, of dimension given by the right hand side of \eqref{dim Zar tang sp a} and containing the above constructed
point $[E]$. Hence the closure $\calm_n$ of $\calm_n(E)$ in
$\calb(-1,n)$ is an irreducible component of $\calb(-1,n)$.
The set $\Sigma_1$ of these components $\calm_n$ is an infinite series distinct from the series 
$\{\calb_0(-1,n)\}_{n\ge1}$ and from the series of Ein components described in \cite{Ein}. 
\end{theorem}

The proof of this Theorem is completely parallel to the proof of Theorem \ref{Thm A}, with clear modifications due to the change from $c_1(E)=0$ to  $c_1(E)=-1$.

It is easy to check that the dimension $\dim\calm_n$ given by \eqref{dim Zar tang sp a}, with $m,\varepsilon$ and $a$ 
as in Theorem \ref{Thm B}, satisfies the strict inequality $\dim\calm_n>8n-5=\dim\calb_0(-1,n)$ (cf. \eqref{dim B0}). 
This shows that $\Sigma_1$ is distinct from $\{\calb_0(-1,n)\}_{n\ge1}$. To distinguish $\Sigma_1$ from
from the series of Ein components, it is enough to see that
the spectra of general bundles of these two series are different. (A direct verification of this fact is left to the reader.) 

\begin{remark}\label{Rem B1}
Let $\caln$ be the set of all values of $n$ for which $\calm_n\in\Sigma_1$, i. e.
$$
\caln=\{n\in2\mathbb{Z}_+\ |\ n=4m+2\varepsilon+a(a+1),\ where\ m\in\mathbb{Z}_+,\ \varepsilon\in\{0,1\},\ a\ge2m+\varepsilon+3 \},
$$
Then one easily sees that
$$
\lim\limits_{r\to\infty}\frac{\caln\cap\{2,4,...,2r\}}{r}=1.
$$
\end{remark}

\vspace{5mm}

\section{Examples of moduli components of stable vector bundles with small values of $c_2$}\label{section 5}

The conditions imposed on the data $(m,\varepsilon,a)$ in Theorem \ref{Thm A}, respectively, Theorem \ref{Thm B} may not be satisfied for small values of these data. However, for some of small values of $(m,\varepsilon,a)$ the equalities \eqref{vanish 1}, \eqref{vanish 2}, \eqref{vanish 2a}, respectively, \eqref{vanish 1a}, \eqref{vanish 2a a}, \eqref{vanish 2b} are still true. 
Hence, our construction of irreducible components $\calm_n\in\Sigma_0$, where $n=2m+\varepsilon+a^2$, respectively,$\calm_n\in\Sigma_1$, where $n=4m+2\varepsilon+a(a+1)$, 
given in Sections \ref{section 3} and \ref{section 4} is still true for these values of $(m,\varepsilon,a)$. 
A precise computation of these values is performed via using the Serre construction \eqref{trC}, respectively, \eqref{Serre-1} for the pairs $([\cale_1],[\cale_2])$ from \eqref{2 inst}, respectively, from \eqref{2 inst a}. We thus provide the following list of irreducible components $\calm_n\in\Sigma_0$ for $n\le20$ and, respectively, $\calm_n\in\Sigma_1$ for $n\le40$.

\subsection{Components $\calm_n\in\Sigma_0$ for $n\le20$}
By $\mathrm{Spec}(E)$ we denote the spectrum
of a general bundle $E$ from $\calm_n$. (Below we use a standard notation $\mathrm{Spec}(E)=
(a^p,b^q,...)$ for the spectrum ($\underset{p}
{\underbrace{a......a}},~\underset{q}
{\underbrace{b......b}},...$).)

(1) $n=6,\ (m,\varepsilon,a)=(1,0,2)$. $\calm_6$ is a component of the expected (by the deformation theory) dimension $\dim\calm_6=45$, and  
$\mathrm{Spec}(E)=(-1,0^4,1)$. This corresponds to the case
6(2) of the Table 5.3 of Hartshorne-Rao \cite{HR}.

(2) $n=7,\ (m,\varepsilon,a)=(1,1,2)$. $\calm_7$ is a component of the expected dimension $\dim\calm_7=53$, and  
$\mathrm{Spec}(E)=(-1,0^5,1)$ (cf. \cite[Table 5.3, 7(2)]{HR}). 

(3) $n=8,\ (m,\varepsilon,a)=(2,0,2)$. $\calm_8$ is a component of the expected dimension $\dim\calm_8=61$, and  
$\mathrm{Spec}(E)=(-1,0^6,1)$ (cf. \cite[Table 5.3, 8(2)]{HR}). 

(4) $n=9,\ (m,\varepsilon,a)=(2,1,2)$. $\calm_9$ is a component of the expected dimension $\dim\calm_9=69$, and  
$\mathrm{Spec}(E)=(-1,0^7,1)$. 

(5) $n=10,\ (m,\varepsilon,a)=(3,0,2)$. $\calm_{10}$ is a component of the expected dimension $\dim\calm_{10}=77$, and  
$\mathrm{Spec}(E)=(-1,0^8,1)$. 

(6) $n=11,\ (m,\varepsilon,a)=(3,1,2)$. $\calm_{11}$ is a component of the expected dimension $\dim\calm_{11}=85$, and  
$\mathrm{Spec}(E)=(-1,0^9,1)$. 

(7) $n=12,\ (m,\varepsilon,a)=(4,0,2)$. $\calm_{12}$ is a component of the expected dimension $\dim\calm_{12}=93$, and  
$\mathrm{Spec}(E)=(-1,0^{10},1)$. 

(8) $n=18,\ (m,\varepsilon,a)=(1,0,4)$. $\calm_{18}$ is a component of the expected dimension $\dim\calm_{12}=141$, and  
$\mathrm{Spec}(E)=(-3,-2^2,-1^3,0^6,1^3,2^2,3)$. 

\subsection{Components $\calm_n\in\Sigma_1$ for $n\le40$}
${}$\\
(1) $n=24,\ (m,\varepsilon,a)=(1,0,4)$. $\calm_{24}$ is a component of the expected dimension $\dim\calm_{24}=187$, and  
$\mathrm{Spec}(E)=(-4,-3^2,-2^3,-1^6,0^6,1^3,2^2,3)$. 

(2) $n=34,\ (m,\varepsilon,a)=(1,0,5)$. $\calm_{34}$ is a component of dimension $\dim\calm_{34}=281$ larger than expected, and  
$\mathrm{Spec}(E)=(-5,-4^2,-3^3,-2^4,-1^7,0^7,1^4,2^3,3^2,4)$. 

(3) $n=36,\ (m,\varepsilon,a)=(1,1,5)$. $\calm_{36}$ is a component of dimension $\dim\calm_{34}=281$ larger than expected, and  
$\mathrm{Spec}(E)=(-5,-4^2,-3^3,-2^4,-1^8,0^8,1^4,2^3,3^2,4)$. 

(4) $n=38,\ (m,\varepsilon,a)=(2,0,5)$. $\calm_{38}$ is a component of the expected dimension $\dim\calm_{36}=299$, and  
$\mathrm{Spec}(E)=(-5,-4^2,-3^3,-2^4,-1^9,0^9,1^4,2^3,3^2,4)$.


\end{document}